\documentclass[pdflatex,sn-mathphys-num]{sn-jnl}

\usepackage{graphicx}%
\usepackage{multirow}%
\usepackage{amsmath,amssymb,amsfonts}%
\usepackage{amsthm}%
\usepackage{mathrsfs}%
\usepackage[title]{appendix}%
\usepackage{xcolor}%
\usepackage{textcomp}%
\usepackage{manyfoot}%
\usepackage{booktabs}%
\usepackage{algorithm}%
\usepackage{algorithmicx}%
\usepackage{algpseudocode}%
\usepackage{listings}%

\usepackage{tikz}
\usepackage{tikz-cd}

\theoremstyle{thmstyleone}%
\newtheorem{theorem}{Theorem}
\newtheorem{proposition}[theorem]{Proposition}%

\theoremstyle{thmstyletwo}%

\newtheorem{lemma}{Lemma}

\theoremstyle{thmstylethree}%
\newtheorem{definition}{Definition}%

\begin{document}

\title[Article Title]{Tensor products, internal homs, and model structures in two dimensional category theory}


\author*[1]{\fnm{Johnathon} \sur{Taylor}}\email{jt3theend17@gmail.com}

\affil*[1]{\orgdiv{Unaffiliated}, \orgaddress{\city{Chicago}, \state{Illinois}, \country{United States}}}


\abstract{In this paper, we introduce a new symmetric monoidal structure on
$\mathbf{Cat}$, called the \emph{graph tensor product}, with unit
given by the terminal category. This tensor product falls in the middle of a factorization between the funny tensor product
and the Cartesian product, giving a factorization connecting these two
classical monoidal structures. We extend this construction to a symmetric
monoidal structure on $2\mathbf{Cat}$, again with unit $D^0$, which
provides an analogous factorization between the funny tensor product and
the Cartesian product of $2$-categories. Using the
$(\mathrm{bo},\mathrm{lff})$ factorization system on $2\mathbf{Cat}$, we
construct a new symmetric monoidal closed model structure on
$2\mathbf{Cat}$ whose tensor product restricts to the Cartesian product
on the subcategory of flexible $2$-categories. Finally, we prove that
this symmetric monoidal model structure fits into a square of weak
symmetric monoidal Quillen equivalences relating the Gray tensor product and
the flexible tensor product.}




\maketitle

\section{Introduction}
Recently, Campbell \cite{Campbell2026} established the existence of a
category $\mathbf{Flex}$ equipped with a cartesian closed model category
structure whose fibrant objects form a category equivalent to the
category of bicategories and pseudofunctors. As Campbell notes, this
provides a convenient base of enrichment for higher category theory.
However, we believe that several aspects of this construction remain
incomplete. The goal of this paper is to further develop these ideas by
constructing and comparing additional monoidal structures on
$2\mathbf{Cat}$ and $\mathbf{Flex}$.

In Section 1, we recall the two biclosed monoidal structures on
$\mathbf{Cat}$ as background. We then introduce the
\emph{graph tensor product} of categories. This tensor product provides an
intermediate structure between the two biclosed monoidal structures,
namely the funny tensor product and the Cartesian product. Moreover, when
restricted to the subcategory $\mathbf{Free}$ of free categories, the
graph tensor product agrees with the Cartesian product.

In Section 2, we recall the $(\mathrm{bo},\mathrm{lff})$ factorization
system on $2\mathbf{Cat}$. We then describe the Cartesian product and the
funny tensor product of $2$-categories and introduce the
\emph{graph tensor product} of $2$-categories. When restricted to the
essential image of the free functor from reflexive $2$-graphs to
$2\mathbf{Cat}$, the graph tensor product agrees with the Cartesian
product. We recover the result of \cite{BourkeGurski} that the Gray
tensor product arises as the middle object of a factorization associated
to the $(\mathrm{bo},\mathrm{lff})$ factorization system on
$2\mathbf{Cat}$. Furthermore, we introduce the
\emph{flexible tensor product} of $2$-categories as another intermediate
object in this factorization.

We paint out the details of an internal hom $[-,-]_f$, inspired by Lack in \cite{Lack2006}, on $2\mathbf{Cat}$ and prove the following theorem.

\begin{center}
 \emph{Theorem \ref{2_cat_with_flex_prod_is_sym_mon_closed}: The category
$2\mathbf{Cat}$ equipped with the flexible tensor product is a symmetric
monoidal closed category with internal hom $[-,-]_f$.}   
\end{center}

We show in Theorem \ref{restrict_of_flex_prod_is_cart_prod} that the
flexible tensor product restricts to the Cartesian product on the
subcategory $\mathbf{Flex}$ of flexible $2$-categories. Furthermore, we
provide a new characterization of the internal hom introduced by
Campbell.

In Section 3, we recall Lack's model structure on $2\mathbf{Cat}$
\cite{Lack2002} and the two monoidal model structures on
$\mathbf{Flex}$ constructed by Campbell \cite{Campbell2026}. We prove in
Theorem \ref{flex_provides_mon_mod_structr} that $2\mathbf{Cat}$,
equipped with Lack's model structure and the flexible tensor product, is
a monoidal model category. Finally, we prove in
Theorem \ref{id_ad_is_weak_Quill_equiv} and Theorem
\ref{adj_is_Quill_equiv} that these constructions assemble into a square
of weak symmetric monoidal Quillen equivalences.

\section{Tensors product of small categories}

\subsection*{The biclosed monoidal products}

There are two biclosed monoidal structures on the category
$\mathbf{Cat}$ of small categories: the Cartesian product and the funny
tensor product. The cartesian product is the categorical product in
$\mathbf{Cat}$, while the funny tensor product is obtained by freely
combining the morphisms of two categories without imposing the
interchange relation. These are the only biclosed monoidal structures on
$\mathbf{Cat}$ \cite{FoltzLairKelly1980}, and they play a fundamental
role in the theory of enriched categories.

The \emph{cartesian product} of categories $A$ and $B$, denoted
$A\times B$, has objects
\[
\operatorname{ob}(A\times B)
=
\operatorname{ob}(A)\times\operatorname{ob}(B),
\]
and hom-sets
\[
(A\times B)\bigl((a,b),(a',b')\bigr)
=
A(a,a')\times B(b,b').
\]
Composition and identities are defined componentwise. Equivalently,
$A\times B$ is generated by the horizontal and vertical copies of the
morphisms of $A$ and $B$, subject to the interchange relation
\[
(f,1_{b'})\circ(1_a,g)
=
(1_{a'},g)\circ(f,1_b),
\]
for every pair of morphisms $f:a\to a'$ in $A$ and
$g:b\to b'$ in $B$.

The \emph{funny tensor product} of categories $A$ and $B$, denoted
$A\star B$, has the same objects as $A\times B$. It is generated by
morphisms
\[
l_{a,g}:(a,b)\longrightarrow(a,b')
\]
for each morphism $g:b\to b'$ in $B$, and
\[
r_{f,b}:(a,b)\longrightarrow(a',b)
\]
for each morphism $f:a\to a'$ in $A$, subject only to the relations
encoding the identities and compositions in $A$ and $B$:
\[
l_{a,1_b}=1_{(a,b)},
\qquad
l_{a,g'g}=l_{a,g'}\circ l_{a,g},
\]
and
\[
r_{1_a,b}=1_{(a,b)},
\qquad
r_{f'f,b}=r_{f',b}\circ r_{f,b}.
\]
Unlike the cartesian product, no interchange relation is imposed between
the left and right generators. Consequently, for morphisms
$f:a\to a'$ and $g:b\to b'$, the composites
\[
r_{f,b'}\circ l_{a,g}
\qquad\text{and}\qquad
l_{a',g}\circ r_{f,b}
\]
are generally distinct.
\subsection*{The graph product}
We use this subsection to define the \emph{graph product} of categories.
\begin{definition}
Let $A$ and $B$ be categories. The \emph{graph product}
$A\diamond B$ is the category obtained from the funny tensor
product $A\star B$ by freely adjoining, for every pair of morphisms
$f:a\to a'$ in $A$ and $g:b\to b'$ in $B$, a new generating morphism
\[
\delta_{f,g}:(a,b)\longrightarrow(a',b').
\]
Equivalently, for every square in $A\star B$
\[
\begin{tikzpicture}[baseline=(current bounding box.center),>=stealth,scale=1]
\node (00) at (0,0) {$(a,b)$};
\node (10) at (2.5,0) {$(a',b)$};
\node (01) at (0,2) {$(a,b')$};
\node (11) at (2.5,2) {$(a',b')$};

\draw[->] (00) -- node[below] {$(f,1)$} (10);
\draw[->] (10) -- node[right] {$(1,g)$} (11);
\draw[->] (00) -- node[left] {$(1,g)$} (01);
\draw[->] (01) -- node[above] {$(f,1)$} (11);

\draw[->] (00) -- node[sloped,above] {$\delta_{f,g}$} (11);
\end{tikzpicture}
\]
one freely adjoins the diagonal morphism $\delta_{f,g}$ and satisfies the unit relations
\[
\delta_{f,\mathrm{id}_b}=r_{f,b},
\qquad
\delta_{\mathrm{id}_a,g}=l_{a,g}.
\]
\end{definition}

The category $A\diamond B$ is generated by the following morphisms.

\begin{itemize}
    \item For every object $a\in A$ and morphism $g:b\to b'$ in $B$, a generator
    \[
    l_{a,g}:(a,b)\to(a,b').
    \]

    \item For every morphism $f:a\to a'$ in $A$ and object $b\in B$, a generator
    \[
    r_{f,b}:(a,b)\to(a',b).
    \]

    \item For every pair of morphisms
    \[
    f:a\to a',
    \qquad
    g:b\to b',
    \]
    a diagonal generator
    \[
    \delta_{f,g}:(a,b)\to(a',b').
    \]
\end{itemize}

The generators satisfy the relations
\begin{align*}
l_{a,\mathrm{id}_b}&=\mathrm{id}_{(a,b)},&
l_{a,hg}&=l_{a,h}\circ l_{a,g},\\
r_{\mathrm{id}_a,b}&=\mathrm{id}_{(a,b)},&
r_{hf,b}&=r_{h,b}\circ r_{f,b},
\end{align*}
for all composable morphisms \(g,h\) in \(B\) and \(f,h\) in \(A\). The diagonal generators \(\delta_{f,g}\)
satisfy the following unit relations
\[
\delta_{f,\mathrm{id}_b}=r_{f,b},
\qquad
\delta_{\mathrm{id}_a,g}=l_{a,g}.
\]

Given functors $F:A\to A'$ and $G:B\to B'$, define
\[
F\diamond G:A\diamond B\longrightarrow A'\diamond B'
\]
by
\[
(F\diamond G)(a,b)=(F(a),G(b))
\]
on objects and by
\begin{align*}
(F\diamond G)(l_{a,g})
&=l_{F(a),G(g)},\\
(F\diamond G)(r_{f,b})
&=r_{F(f),G(b)},\\
(F\diamond G)(\delta_{f,g})
&=\delta_{F(f),G(g)}
\end{align*}
on generating morphisms. The defining relations are preserved, and hence
$F\diamond G$ extends uniquely to a functor.

\begin{definition}
Let $A$, $B$, and $C$ be categories. The associativity morphism
\[
\alpha_{A,B,C}:(A\diamond B)\diamond C
\longrightarrow
A\diamond(B\diamond C)
\]
is the functor determined by the following assignments on generators.

On objects, we define
\[
\alpha_{A,B,C}(a,b,c)=(a,b,c).
\]

For the generating morphisms of $(A\diamond B)\diamond C$, we define
\begin{align*}
\alpha_{A,B,C}(l_{(a,b),h})
&=l_{a,l_{b,h}},\\
\alpha_{A,B,C}(r_{l_{a,g},c})
&=l_{a,r_{g,c}},\\
\alpha_{A,B,C}(r_{r_{f,b},c})
&=r_{f,(b,c)},\\
\alpha_{A,B,C}(r_{\delta_{f,g},c})
&=\delta_{f,r_{g,c}},\\
\alpha_{A,B,C}(\delta_{l_{a,g},h})
&=l_{a,\delta_{g,h}},\\
\alpha_{A,B,C}(\delta_{r_{f,b},h})
&=\delta_{f,l_{b,h}},\\
\alpha_{A,B,C}(\delta_{\delta_{f,g},h})
&=\delta_{f,\delta_{g,h}}.
\end{align*}
Here $f:a\to a'$ is a morphism in $A$, $g:b\to b'$ is a morphism in
$B$, and $h:c\to c'$ is a morphism in $C$.
\end{definition}

\begin{lemma}\label{lemma_ass_forms_nat_iso}
The associativity morphisms form a natural isomorphism $\alpha$.
\end{lemma}

\begin{proof}
For all categories $A$, $B$, and $C$, $\alpha_{A,B,C}$ is invertible because the list of generators for $(A\diamond B)\diamond C$ are bijectively mapped onto the generators for $A\diamond(B\diamond C)$. Let $F:A\to A'$, $G:B\to B'$, and $H:C\to C'$ be functors. Then 
\[
\alpha_{A',B',C'}\circ[ (F\diamond G)\diamond H])=[F\diamond(G\diamond H)]\circ\alpha_{A,B,C}
\]
by construction and definition. Therefore $\alpha$ is a natural isomorphism.
\end{proof}

\begin{definition}
Let $D^0$ denote the terminal category with unique object $*$. The
left and right unitors are the functors
\[
\lambda_A:D^0\diamond A\longrightarrow A,
\qquad
\rho_A:A\diamond D^0\longrightarrow A,
\]
defined on objects by
\[
\lambda_A(*,a)=a,
\qquad
\rho_A(a,*)=a.
\]

On generating morphisms, the left unitor is given by
\begin{align*}
\lambda_A(l_{*,g})&=g,\\
\lambda_A(r_{\mathrm{id}_*,a})&=\mathrm{id}_a,\\
\lambda_A(\delta_{\mathrm{id}_*,g})&=g,
\end{align*}
where $g:a\to a'$ is a morphism in $A$.

The right unitor is given by
\begin{align*}
\rho_A(r_{f,*})&=f,\\
\rho_A(l_{a,\mathrm{id}_*})&=\mathrm{id}_a,\\
\rho_A(\delta_{f,\mathrm{id}_*})&=f,
\end{align*}
where $f:a\to a'$ is a morphism in $A$.

These assignments preserve the defining relations of the graph product and
therefore extend uniquely to functors.
\end{definition}

\begin{lemma}
The left and right unity morphisms form natural isomorphisms $\lambda$ and $\rho$, respectively.
\end{lemma}

\begin{proof}
Once again, the list of generators of the source are mapped bijectively onto the generators of the target.
\end{proof}

\begin{definition}
Let $A$ and $B$ be categories. The symmetry morphism is the functor
\[
\sigma_{A,B}:A\diamond B\longrightarrow B\diamond A
\]
defined on objects by
\[
\sigma_{A,B}(a,b)=(b,a).
\]

On generating morphisms, the symmetry is given by
\begin{align*}
\sigma_{A,B}(l_{a,g})
&=r_{g,a},\\
\sigma_{A,B}(r_{f,b})
&=l_{b,f},\\
\sigma_{A,B}(\delta_{f,g})
&=\delta_{g,f}.
\end{align*}
These assignments preserve the defining relations of the graph product and
therefore extend uniquely to a functor.
\end{definition}

\begin{lemma}
The symmetry morphisms form a natural isomorphism $\sigma$.
\end{lemma}

\begin{proof}
Just as before, the list of generators of the source are mapped bijectively onto the generators of the target.
\end{proof}

\begin{theorem}
The category of small categories $\mathbf{Cat}$ equipped with the graph product is a symmetric monoidal category.
\end{theorem}
We shall not provide the proof here. The proof follows from a diagram chase implemented on the generators. 

Let $\textbf{Free}$ be the category of free categories and free morphisms \cite{Campbell2026}. When we restrict to $\mathbf{Free}$, the graph product is the cartesian product inherited through an equivalence with reflexive graphs. We have just shown that we may extend its structure to all of $\textbf{Cat}$. Whereas the graph product does not form a biclosed monoidal structure on $\textbf{Cat}$, it interestingly does form a biclosed monoidal structure on $\textbf{Free}$.

\subsection*{Comparison of tensors}
We now provide a short subsection to compare the tensor products. There is a canonical factorization of tensor products of categories
\[
A\star B
\longrightarrow
A\diamond B
\longrightarrow
A\times B.
\]
The first functor freely adjoins the diagonal morphisms
$\delta_{f,g}$, while the second quotient identifies these diagonal
morphisms with the canonical composites in the cartesian product. Let $D^1$ be the walking morphism. The funny tensor product, the graph tensor product, and the cartesian product of $D^1$ with itself our displayed in the order we just listed, respectively.
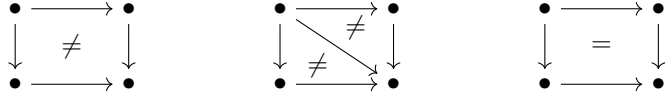
\[
\begin{tikzpicture}

\node (a1) at (0,1) {$\bullet$};
\node (b1) at (1.5,1) {$\bullet$};
\node (c1) at (0,0) {$\bullet$};
\node (d1) at (1.5,0) {$\bullet$};

\draw[->] (a1)--(b1);
\draw[->] (a1)--(c1);
\draw[->] (b1)--(d1);
\draw[->] (c1)--(d1);

\node at (0.75,0.5) {$\neq$};

\node (a2) at (3.5,1) {$\bullet$};
\node (b2) at (5,1) {$\bullet$};
\node (c2) at (3.5,0) {$\bullet$};
\node (d2) at (5,0) {$\bullet$};

\draw[->] (a2)--(b2);
\draw[->] (a2)--(c2);
\draw[->] (b2)--(d2);
\draw[->] (c2)--(d2);

\draw[->] (a2)--(d2);

\node at (4.5,0.75) {$\neq$};

\node at (4,0.25) {$\neq$};

\node (a3) at (7,1) {$\bullet$};
\node (b3) at (8.5,1) {$\bullet$};
\node (c3) at (7,0) {$\bullet$};
\node (d3) at (8.5,0) {$\bullet$};

\draw[->] (a3)--(b3);
\draw[->] (a3)--(c3);
\draw[->] (b3)--(d3);
\draw[->] (c3)--(d3);

\node at (7.75,0.5) {$=$};

\end{tikzpicture}
\]

\section{Factorization systems, tensor products, and internal homs in two dimensional category theory}
\subsection*{Factorization systems}
 A factorization system on a category $\mathcal{C}$
is a pair of classes of morphisms
\[
(\mathcal{E},\mathcal{M})
\]
such that every morphism $f:X\to Y$ admits a factorization
\[
X\xrightarrow{e_f}Z_f\xrightarrow{m_f}Y
\]
with
\[
e_f\in\mathcal{E},\qquad m_f\in\mathcal{M},
\]
and satisfying
\[
\mathcal{E}={}^{\perp}\mathcal{M},
\qquad
\mathcal{M}=\mathcal{E}^{\perp}.
\]

\begin{definition}
A factorization system $(\mathcal{E},\mathcal{M})$ on a category $\mathcal{C}$ is \emph{reflective factorization system} if $\mathcal{C}$ has a terminal object $1$ and the maps in $\mathcal{E}$ satisfy 2 for 3.
\end{definition}

\begin{definition}
The \emph{$(\mathrm{bo},\mathrm{lff})$ factorization system} on $2\mathbf{Cat}$ is the
factorization system $(\mathcal{E},\mathcal{M})$ where $\mathcal{E}$ is
the class of $2$-functors that are bijective on objects and bijective on
arrows, and $\mathcal{M}$ is the class of locally full and faithful
$2$-functors.
\end{definition}

Bourke and Gurski shows that the $(\mathrm{bo},\mathrm{lff})$ factorization system on $2\mathbf{Cat}$ is reflective. 
\subsection*{The Basic tensor products of two categories}
\begin{definition}
The \emph{cartesian product} of $2$-categories $A$ and
$B$ is the $2$-category $A\times B$ whose
objects, $1$-morphisms, and $2$-morphisms are given componentwise. More
precisely,
\begin{align*}
\operatorname{ob}(A\times B)
    &=\operatorname{ob}(A)\times\operatorname{ob}(B),\\
(A\times\mathcal B)((a,b),(a',b'))
    &=A(a,a')\times B(b,b'),
\end{align*}
with identities and composition defined componentwise.
\end{definition}

\begin{definition}
The \emph{funny tensor product} of $2$-categories $A$ and
$B$ is the $2$-category $A\star B$ defined
as the pushout
\[
\begin{tikzpicture}[baseline=(current bounding box.center)]
\node (TL) at (0,1.5)
{$\operatorname{ob}(A)\times\operatorname{ob}(B)$};
\node (TR) at (4,1.5)
{$A\times\operatorname{ob}(B)$};
\node (BL) at (0,0)
{$\operatorname{ob}(A)\times B$};
\node (BR) at (4,0)
{$A\star B$};

\draw[->] (TL) -- (TR);
\draw[->] (TL) -- (BL);
\draw[->] (TR) -- (BR);
\draw[->] (BL) -- (BR);
\end{tikzpicture}
\]
in $2\mathbf{Cat}$.
\end{definition}

\begin{definition}
Let $A$ and $B$ be $2$-categories. We define the graph tensor product as follows. We first construct an intermediate pushout 
\[
\begin{tikzpicture}[baseline=(current bounding box.center)]
\node (TL) at (0,1.5)
{$\coprod_{\substack{(f,g)\in A_1\times B_1,\\ f\neq 1,g\neq 1}}S^0$};
\node (TR) at (4,1.5)
{$A\star B$};
\node (BL) at (0,0)
{$\coprod_{\substack{(f,g)\in A_1\times B_1,\\ f\neq 1,g\neq 1}}D^1$};
\node (BR) at (4,0)
{$A\diamond_p B$};

\draw[->] (TL) -- (TR);
\draw[->] (TL) -- (BL);
\draw[->] (TR) -- (BR);
\draw[->] (BL) -- (BR);
\end{tikzpicture}
\]
in $2\mathbf{Cat}$. This freely adds diagonal $1$-cells for every pair $(f,g)$ of non-identity $1$-cell $f$ of $A$ and non-identity $1$-cell $g$ of $B$. We shall label the $1$-cells added by $\delta_{f,g}$ where we have a non-identity $1$-cell $f$ of $A$ and non-identity $1$-cell $g$ of $B$. If either $f=1$ or $g=1$, we will have a morphism $\delta_{f,g}$, but in the case $f=1_a$, we will have $\delta_{1,g}=l_{a,g}$ and in the case $g=1_b$, we will have $\delta_{f,1_b}=r_{f,b}$.
We now define the \emph{graph product} of $A$ and $B$ as the pushout 

\[
\begin{tikzpicture}[baseline=(current bounding box.center)]
\node (TL) at (0,1.5)
{$\coprod_{\substack{(\alpha,\beta)\in A_2\times B_2,\\ \alpha\neq 1_1,\beta\neq 1_1}}S^1$};
\node (TR) at (4,1.5)
{$A\diamond_p B$};
\node (BL) at (0,0)
{$\coprod_{\substack{(\alpha,\beta)\in A_2\times B_2,\\ \alpha\neq 1_1,\beta\neq 1_1}}D^2$};
\node (BR) at (4,0)
{$A\diamond B$};

\draw[->] (TL) -- (TR);
\draw[->] (TL) -- (BL);
\draw[->] (TR) -- (BR);
\draw[->] (BL) -- (BR);
\end{tikzpicture}
\]
in $2\mathbf{Cat}$. We shall label the $2$-cells added as $\delta_{\alpha,\beta}$. If either $\alpha=1_1$ or $\beta=1_1$, we will have a $2$-cell $\delta_{\alpha,\beta}$, but in the case $\alpha=1_{1_a}$, we will have $\delta_{1_{1_a},\beta}=l_{a,\beta}$ and in the case $g=1_{1_b}$, we will have $\delta_{\alpha,1_{1_b}}=r_{\alpha,b}$. Moreover, we we write $i:A\star B\to A\diamond B$ for the $2$-functor inclusion.
\end{definition}

We have now obtained a square of $2$-categories.
\begin{equation}\label{important square}
    \begin{tikzpicture}[baseline=(current bounding box.center)]
\node (TL) at (0,1.5)
{$A\star B$};
\node (TR) at (4,1.5)
{$A\times B$};
\node (BL) at (0,0)
{$A\diamond B$};
\node (BR) at (4,0)
{$A\times B$};

\draw[->] (TL) to node[above]{$\mathbf{st}$} (TR);
\draw[->] (TL) to node[left]{$i$} (BL);
\draw[->] (TR) to node[right]{$1_{A\times B}$} (BR);
\draw[->] (BL) to node[below]{$\mathbf{st}$} (BR);
\end{tikzpicture}
\end{equation}

\subsection*{Tensor products via factorization systems}
We now apply Theorem 1.6 of \cite{BourkeGurski} to the square (\ref{important square}) to produce new tensor products from these old ones.

\begin{definition}
Let $A$ and $B$ be $2$-categories. We define the \emph{Gray tensor product} of $A$ and $B$, denoted by $A\otimes B$, to be the middle of the factorization of the map $\mathbf{st}:A\star B\to A\times B$  in the $(\mathrm{bo},\mathrm{lff})$-factorization system on $2\mathbf{Cat}$.
\end{definition}

That is to say that we obtain maps
\[
A\star B\xrightarrow{P}A\otimes B\xrightarrow{Q}A\times B,
\]
where $P$ is boba and $Q$ is lff. Theorem 3.16 of \cite{BourkeGurski} shows that this definition of the Gray tensor product coincides with the classical definition given by Gray in Chapter I, Section 4 of \cite{Gray1974}.

\begin{definition}
Let $A$ and $B$ be $2$-categories. We define the \emph{flexible tensor product} of $A$ and $B$, denoted by $A\boxtimes B$, to be the middle of the factorization of the map $\mathbf{st}:A\diamond B\to A\times B$ in the $(\mathrm{bo},\mathrm{lff})$-factorization system on $2\mathbf{Cat}$.
\end{definition}

That is to say that we obtain maps
\[
A\diamond B\xrightarrow{R}A\boxtimes B\xrightarrow{S}A\times B,
\]
where $P$ is boba and $Q$ is lff. Moreover, we induce a unique map 
\[
j:A\otimes B\to A\boxtimes B
\]
in $2\mathbf{Cat}$ such that the following diagram commutes.
\[
\begin{tikzpicture}[node distance=2cm]
\node (A) {$A\star B$};
\node (B) [below of=A] {$A\diamond B$};
\node (C) [right of=A] {$A\otimes B$};
\node (D) [below of=C] {$A\boxtimes B$};
\node (E) [right of=C] {$A\times B$};
\node (F) [below of=E] {$A\times B$};

\draw[->] (A) -- node[left] {$i$} (B);
\draw[->] (A) -- node[above] {$P$} (C);
\draw[->] (B) -- node[below] {$R$} (D);
\draw[->] (C) -- node[left] {$j$} (D);
\draw[->] (C) -- node[above] {$Q$} (E);
\draw[->] (D) -- node[below] {$S$} (F);
\draw[->] (E) -- node[right] {$1_{A\times B}$} (F);
\end{tikzpicture}
\]

\begin{lemma}\label{biequival_between_products}
Given $2$-categories $A$ and $B$, the $2$-functors 
\[
j:A\otimes B\to A\boxtimes B
\]
\[
S:A\boxtimes B\to A\times B
\]
are biequivalences.
\end{lemma}

\begin{proof}
   Notice that $j$ is a bijection on $0$-cells since $P$, $i$, and $R$ are. Moreover, $j$ is surjective on $1$-cells, locally full, and faithful since $Q$, $1_{A\times B}$, and $S$ are. Therefore $j$ is a biequivalence. We have that $Q$ is always a biequivalence by \cite{Lack2002}. As $Q$, $1_{A\times B}$, and $j$ are biequivalences, so is $S$.
\end{proof}

Let $A$ and $B$ both be copies of the $2$-category with precisely one non-invertible $1$-cell. Visually speaking we obtain the following two squares in the Gray tensor product, flexible tensor product, and cartesian product, respectively.
\[
\begin{tikzpicture}

\node (a1) at (0,1) {$\bullet$};
\node (b1) at (1.5,1) {$\bullet$};
\node (c1) at (0,0) {$\bullet$};
\node (d1) at (1.5,0) {$\bullet$};

\draw[->] (a1) --(b1);
\draw[->] (a1)--(c1);
\draw[->] (b1)--(d1);
\draw[->] (c1)--(d1);

\node at (0.75,0.5) {$\cong$};

\node (a2) at (3.5,1) {$\bullet$};
\node (b2) at (5,1) {$\bullet$};
\node (c2) at (3.5,0) {$\bullet$};
\node (d2) at (5,0) {$\bullet$};

\draw[->] (a2)--(b2);
\draw[->] (a2)--(c2);
\draw[->] (b2)--(d2);
\draw[->] (c2)--(d2);

\draw[->] (a2)--(d2);

\node at (4.5,0.75) {$\cong$};

\node at (4,0.25) {$\cong$};

\node (a3) at (7,1) {$\bullet$};
\node (b3) at (8.5,1) {$\bullet$};
\node (c3) at (7,0) {$\bullet$};
\node (d3) at (8.5,0) {$\bullet$};

\draw[->] (a3)--(b3);
\draw[->] (a3)--(c3);
\draw[->] (b3)--(d3);
\draw[->] (c3)--(d3);

\node at (7.75,0.5) {$=$};

\end{tikzpicture}
\]

The following lemma is immediate from construction.

\begin{lemma}\label{uni_prop_of_flex_prod}
A strict $2$-functor $F:A\boxtimes B\to C$ is determined by the composite
\[
A\otimes B\xrightarrow{j}A\boxtimes B\xrightarrow{F}C
\]
and a choice of $2$-cell isomorphism
\[
\xi_{f,g}:F(\delta_{f,g})\Rightarrow F(f,b')\circ F(a,g)
\]
for every non-identity $1$-cell $f$ of $A$ and every non-identity $1$-cell $g$ of $B$.
\end{lemma}

\subsection*{Internal Homs}
We now discuss symmetric monoidal closed structures on $2\mathbf{Cat}$. We start by providing background on the two most used symmetric monoidal closed structures on $2\mathbf{Cat}$.

\begin{theorem}\label{cart_product}
The category $2\mathbf{Cat}$ of $2$-categories and strict functors together with the Cartesian
product forms a symmetric monoidal category.
\end{theorem}

When we consider the cartesian product, the internal hom has strict $2$-functors, strict transformations, and modifications for $0$-, $1$-, and $2$-cells, respectively. 

\begin{theorem}\label{Gray_product}
The category $2\mathbf{Cat}$ of $2$-categories and strict functors together with the Cartesian
product forms a symmetric monoidal category.
\end{theorem}

For the Gray tensor product, the internal hom has strict $2$-functors, pseudonatural transformations, and modifications for $0$-, $1$-, and $2$-cells, respectively.

We write $[-,-]_\times$ and $[-,-]_g$ for the internal homs for the cartesian product and Gray tensor product, respectively. We shall now show that $2\mathbf{Cat}$ equipped with the flexible tensor product is closed. Let $A$ and $B$ be $2$-categories. We now define a new hom that is associated to the flexible tensor product. We define $[A,B]_f$ to have strict $2$-functors for $0$-cells, the $1$-cells are enhanced pseudonatural transformations of Lack from pp 6-7 of \cite{Lack2006}, and the $2$-cells are the modifications between the enhanced pseudonatural transformations.

\begin{theorem}\label{2_cat_with_flex_prod_is_sym_mon_closed}
    The category $2\mathbf{Cat}$ equipped with the flexible tensor product is a symmetric monoidal closed category with internal hom $[-,-]_f$.
\end{theorem}

\begin{proof}
We shall prove that there is a bijection
\[
\phi:2\mathbf{Cat}(A\boxtimes B, C)\to 2\mathbf{Cat}(A,[B,C]_f)
\]
natural in $A$, $B$, and $C$. Let $F:A\boxtimes B\to C$ be a strict $2$-functor. We first define $\phi(F)(a)$ to be the strict $2$-functor
\[
\phi(F)(a)=F(a,-)
\]
for all $0$-cells $a$. Given a $1$-cell $f:a\to a'$, we define 
\[
\phi(F)(f):F(a,-)\to F(a',)
\]
to be the enhanced pseudo-natural transformation defined to have underlying pseudo-natural transformation given by $F(f,-)$ and additional $1$-cell 
\[
F(\delta_{f,g}):F(a,b)\to F(a',b')
\]
equipped with a $2$-cell isomorphism 

\[
F(\xi_{f,g}): F(\delta_{f,g})\Rightarrow F(f,b')\circ F(a,g)
\]
for all $1$-cells $g:b\to b'$ in $B$. Given a $2$-cell $\chi:f\to f'$, we define $\phi(F)(\chi)$ to  be the modification
\[
\phi(F)(\chi)=F(\chi,-).
\]
Naturality and bijectivity follows from construction, a diagram chase that invokes Lemma \ref{uni_prop_of_flex_prod}, and using the fact that $(\otimes, [-,-]_g)$ is a tensor-hom pair on $2\mathbf{Cat}$.
\end{proof}

\subsection{Restriction to Flexible Two Categories}
Write $\mathbf{Flex}$ to denote the category of flexible $2$-categories introduced in \cite{Campbell2026}. There is a forgetful functor $U:\mathbf{Flex}\to 2\mathbf{Cat}$ that is left adjoint to the path $2$-category functor $Q:2\mathbf{Cat}\to\mathbf{Flex}$. Campbell showed the funny tensor product $\star$ and the Gray tensor product $\otimes$ restricts to $\mathbf{Flex}$. We shall know show that the graph product $\diamond$ and the flexible tensor product $\boxtimes$ restricts to $\mathbf{Flex}$. Moreover, we show when we restrict $\boxtimes$ to $\mathbf{Flex}$, it is the cartesian product of $\mathbf{Flex}$ introduced by Campbell.

\begin{proposition}
Let $A$ and $B$ be flexible $2$-categories. Then the graph tensor
product $A\diamond B$ is again an object of $\mathbf{Flex}$.
\end{proposition}

\begin{proof}
Let
\[
U:2\mathbf{Cat}\longrightarrow \mathbf{Cat}
\]
be the functor which sends a $2$-category to its underlying
category. Since $A$ and $B$ are flexible $2$-categories, their underlying
categories $U(A)$ and $U(B)$ are free categories. The graph product of the underlying categories therefore gives a free
category
\[
U(A)\diamond U(B),
\]
as it is the cartesian product on $\textbf{Free}$. There is a canonical morphism
\[
U(A)\diamond U(B)\longrightarrow U(A)\times U(B)\cong U(A\times B)
\]
determined on generators by assigning
\begin{align*}
l_{a,g}&\longmapsto (1_a,g),\\
r_{f,b}&\longmapsto (f,1_b),\\
\delta_{f,g}&\longmapsto (f,g).
\end{align*}

We now use the fact that $U$ is a Grothendieck bifibration and notice that the graph tensor product $A\diamond B$ in this case is the
$U$-cocartesian lift of this morphism. We may conclude that $A\diamond B$ is again a flexible
$2$-category.
\end{proof}

We provided a proof that is essentially dual usage to how Campbell defined the cartesian product for flexible $2$-categories. Speaking of that, we now show that flexible product restricts to $\mathbf{Flex}$ and is the cartesian product.

\begin{theorem}\label{restrict_of_flex_prod_is_cart_prod}
Let $A$ and $B$ be flexible $2$-categories. Then $A\boxtimes B$ is a flexible $2$-category and is the cartesian product of $\mathbf{Flex}$.
\end{theorem}

\begin{proof}
As $A\diamond B$ is a flexible $2$-category, $A\boxtimes B$ is forced to be as well since 
\[
R:A\diamond B\to A\boxtimes B
\]
is a boba map in $2\mathbf{Cat}$. As $S:A\boxtimes B\to A\times B$ is locally fully faithful, it follows from the paragraph preceding the construction of $\mathbf{Flex}$ in \cite{Campbell2026} that $A\boxtimes B$ is the $U$-cartesian lift of the morphism 
\[
U(A)\diamond U(B)\cong U(A\boxtimes B)\xrightarrow{U(S)} U(A\times B)
\]
in $2\mathbf{Cat}$. Therefore $A\boxtimes B$ is the cartesian product of flexible $2$-categories as defined by Campbell. 
\end{proof}

We now define the internal homs inherited. Let $A$ and $B$ be flexible $2$-categories. Campbell defined the internal hom $[A,B]_g$ on $\mathbf{Flex}$ with respect to the Gray tensor product as the $U$-cartesian lift to the top horizontal of the following pullback.

\begin{equation}
    \begin{tikzpicture}[baseline=(current bounding box.center)]
\node (TL) at (0,1.5)
{$U[A,B]_\mathbf{Flex}$};
\node (TR) at (4,1.5)
{$U[A,B]_{g}$};
\node (BL) at (0,0)
{$[UA,UB]_{\textbf{Free}}$};
\node (BR) at (4,0)
{$[UA,UB]_{\mathbf{Cat}}$};

\draw[->] (TL) to node[above]{} (TR);
\draw[->] (TL) to node[left]{} (BL);
\draw[->] (TR) to node[right]{} (BR);
\draw[->] (BL) to node[below]{} (BR);
\end{tikzpicture}
\end{equation}

We now define the internal hom $\mathbf{Flex}(-,-)$ on $\mathbf{Flex}$ with respect to the flexible tensor product as the $U$-cartesian lift to the top horizontal of the following pullback.

\begin{equation}
    \begin{tikzpicture}[baseline=(current bounding box.center)]
\node (TL) at (0,1.5){$U(\mathbf{Flex}(A,B))$};
\node (TR) at (4,1.5)
{$U[A,B]_{f}$};
\node (BL) at (0,0)
{$[UA,UB]_{\textbf{Free}}$};
\node (BR) at (4,0)
{$[UA,UB]_{\mathbf{Cat}}$};

\draw[->] (TL) to node[above]{} (TR);
\draw[->] (TL) to node[left]{} (BL);
\draw[->] (TR) to node[right]{} (BR);
\draw[->] (BL) to node[below]{} (BR);
\end{tikzpicture}
\end{equation}

The definition we gave here is equivalent to the definition provided by Campbell.

\section{Homotopy theory of Two Dimensional Structures}

\subsection*{Model structures}
We adopt the following notation borrowed from Lack \cite{Lack2002} throughout.

\begin{itemize}
    \item Let $\mathbf{E}$ denote the free-living equivalence with objects
    $x$ and $y$, $1$-morphisms
    \[
    s:x\to y,\qquad t:y\to x,
    \]
    and invertible $2$-morphisms
    \[
    ts\Rightarrow 1_x,
    \qquad
    st\Rightarrow 1_y.
    \]

    \item Let $2$ denote the discrete $2$-category on two objects.

    \item Let $\mathbf{2}$ denote the free $2$-category generated by a
    single $1$-morphism
    \[
    x\longrightarrow y.
    \]

    \item Let $\mathbf{D}$ denote the $2$-category obtained from
    $\mathbf{2}$ by freely adjoining a second parallel $1$-morphism
    together with an invertible $2$-morphism between them.

    \item Let $C_2$ denote the $2$-category with objects $x$ and $y$,
    non-identity $1$-morphisms
    \[
    s,s':x\to y,
    \]
    and non-identity $2$-morphisms
    \[
    \sigma_1,\sigma_2:s\Rightarrow s'.
    \]
    Let $C_1$ denote the sub-$2$-category containing $\sigma_1$ but not
    $\sigma_2$, and let $C_0$ denote the sub-$2$-category containing all
    $1$-morphisms but no non-identity $2$-morphisms.

    \item Let
    \[
    j_1:D^0\longrightarrow\mathbf{E}
    \]
    denote the map selecting the object $x$.

    \item Let
    \[
    j_2:\mathbf{2}\longrightarrow\mathbf{D}
    \]
    denote the map selecting the $1$-morphism $s$.

    \item Let
    \[
    i_1:\emptyset\longrightarrow D^0
    \]
    denote the unique map.

    \item Let
    \[
    i_2:2\longrightarrow\mathbf{2}
    \]
    denote the inclusion.

    \item Let
    \[
    i_3:C_0\longrightarrow C_1
    \]
    denote the inclusion.

    \item Let
    \[
    i_4:C_2\longrightarrow C_1
    \]
    denote the map sending both $\sigma_1$ and $\sigma_2$ to
    $\sigma_1$.

    \item Finally, define
    \[
    I=\{i_1,i_2,i_3,i_4\},
    \qquad
    J=\{j_1,j_2\}.
    \]
\end{itemize}

\begin{theorem}[{\cite[Theorem~6.3]{Lack2006}}]
The category $2\mathbf{Cat}$ admits a proper combinatorial model
structure in which the weak equivalences are the biequivalences, the
generating cofibrations are the maps in $I$, and the generating trivial
cofibrations are the maps in $J$.
\end{theorem}

The model structure here is called \emph{Lack's model structure}. Campbell constructed an analogous model structure on the category of flexible $2$-categories.

\begin{theorem}[{\cite[Theorem~4.9]{Campbell2026}}]
The category $\mathbf{Flex}$ admits a combinatorial model structure in
which a flexible morphism is a weak equivalence if and only if it is a
biequivalence, and is a cofibration if and only if it is injective on
objects and faithful. Moreover, every object of $\mathbf{Flex}$ is
cofibrant.
\end{theorem}

Campbell's model structure is left induced from Lack's model structure
on $2\mathbf{Cat}$ along the forgetful functor
\[
U:\mathbf{Flex}\longrightarrow2\mathbf{Cat}.
\]
Moreover, the adjunction between $\mathbf{Flex}$ and $2\mathbf{Cat}$ is
a Quillen equivalence.

\begin{theorem}[{\cite[Theorem~4.10]{Campbell2026}}]
The adjunction
\[
\begin{tikzpicture}[node distance=4cm]
\node (A) {$\mathbf{Flex}$};
\node (B) [right of=A] {$2\mathbf{Cat}$};

\draw[transform canvas={yshift=0.6ex},->] (A) -- node[above] {$U$} (B);
\draw[transform canvas={yshift=-0.6ex},<-] (A) -- node[below] {$Q$} (B);
\end{tikzpicture}
\]
is a Quillen equivalence.
\end{theorem}

\subsection*{Monoidal Model Structures}
The Gray tensor product provides the standard monoidal structure on
$2\mathbf{Cat}$ that is compatible with Lack's model structure. In addition, the Gray tensor product restricts to make $\mathbf{Flex}$ a symmetric monoidal model category. $\mathbf{Flex}$ has the structure of a monoidal model category when equipped with the flexible tensor product. We now recall these facts.

\begin{theorem}[{\cite[Theorem~7.5]{Lack2002}}]
The category $2\mathbf{Cat}$ equipped with the Lack model structure and
the Gray tensor product is a monoidal model category.
\end{theorem}

\begin{theorem}[{\cite[Theorem~7.3]{Campbell2026}}]
The category $\mathbf{Flex}$ equipped with the left-induced model
structure and the Gray tensor product is a monoidal model category.
\end{theorem}

\begin{theorem}[{\cite[Theorem~7.4]{Campbell2026}}]
The adjunction
\[
\begin{tikzpicture}[node distance=4cm]
\node (A) {$(\mathbf{Flex},\otimes)$};
\node (B) [right of=A] {$(2\mathbf{Cat},\otimes)$};

\draw[transform canvas={yshift=0.6ex},->] (A) -- node[above] {$U$} (B);
\draw[transform canvas={yshift=-0.6ex},<-] (A) -- node[below] {$Q$} (B);
\end{tikzpicture}
\]
is a symmetric monoidal Quillen equivalence.
\end{theorem}

Campbell shows that the flexible tensor product equipped with the left transfer model structure provides the structure of a monoidal model category.

\begin{theorem}[{\cite[Theorem~7.6]{Campbell2026}}]
The category $\mathbf{Flex}$ equipped with the left-induced model
structure and the Gray tensor product is a monoidal model category.
\end{theorem}

\begin{theorem}[{\cite[Theorem~7.13]{Campbell2026}}]
The identity adjunction
\[
\begin{tikzpicture}[node distance=4cm]
\node (A) {$(\mathbf{Flex},\otimes)$};
\node (B) [right of=A] {$(\mathbf{Flex},\boxtimes)$};

\draw[transform canvas={yshift=0.6ex},->] (A) -- node[above] {$1$} (B);
\draw[transform canvas={yshift=-0.6ex},<-] (A) -- node[below] {$1$} (B);
\end{tikzpicture}
\]
is a weak symmetric monoidal Quillen equivalence.
\end{theorem}

The preceding results establish the compatibility of the Gray and
flexible tensor products on $\mathbf{Flex}$. We now transfer this
structure back to $2\mathbf{Cat}$ and compare the resulting monoidal
model structures. The key observation is that the flexible tensor product
is related to the Gray tensor product by biequivalences.

\begin{lemma}
If $i$ and $j$ are cofibrations, then
\[
i\boxtimes j
\]
is a cofibration. In particular, if $A$ is cofibrant and $j$ is a
cofibration, then
\[
A\boxtimes j
\]
is a cofibration.
\end{lemma}

\begin{proof}
The proof is identical to that of \cite[Lemma~7.3]{Lack2002}. 
\end{proof}

Lack's calculations involving the generating cofibrations in the proof of \cite[Lemma~7.3]{Lack2002} are reproduced here verbatim.

\begin{lemma}\label{biequival_of_products}
If $F:A\to B$ is a biequivalence and $C$ is a $2$-category, then
\[
C\boxtimes F:C\boxtimes A\longrightarrow C\boxtimes B
\]
is a biequivalence.
\end{lemma}

\begin{proof}
By Lemma~\ref{biequival_of_products}, the comparison maps
\[
j:C\otimes A\longrightarrow C\boxtimes A
\]
and
\[
j:C\otimes B\longrightarrow C\boxtimes B
\]
are biequivalences. Moreover, $C\otimes F$ is a biequivalence by
\cite[Lemma~7.4]{Lack2002}. Since
\[
j\circ(C\otimes F)
=
(C\boxtimes F)\circ j,
\]
the two-out-of-three property implies that $C\boxtimes F$ is a
biequivalence.
\end{proof}

\begin{theorem}\label{flex_provides_mon_mod_structr}
The Lack model structure on $2\mathbf{Cat}$ together with the flexible
tensor product is a monoidal model category.
\end{theorem}

\begin{proof}
The proof is identical to that of
\cite[Theorem~7.5]{Lack2002}, using the preceding lemmas.
\end{proof}

\begin{theorem}\label{id_ad_is_weak_Quill_equiv}
The identity adjunction
\[
\begin{tikzpicture}[node distance=4cm]
\node (A) {$(2\mathbf{Cat},\otimes)$};
\node (B) [right of=A] {$(2\mathbf{Cat},\boxtimes)$};

\draw[transform canvas={yshift=0.6ex},->] (A) -- node[above] {$1$} (B);
\draw[transform canvas={yshift=-0.6ex},<-] (A) -- node[below] {$1$} (B);
\end{tikzpicture}
\]
is a weak symmetric monoidal Quillen equivalence.
\end{theorem}

\begin{proof}\label{adj_is_Quill_equiv}
This follows from Lemma~\ref{biequival_between_products} and the existence of the model structure due to Lack \cite{Lack2002}.
\end{proof}

\begin{theorem}
The adjunction
\[
\begin{tikzpicture}[node distance=4cm]
\node (A) {$(\mathbf{Flex},\boxtimes)$};
\node (B) [right of=A] {$(2\mathbf{Cat},\boxtimes)$};

\draw[transform canvas={yshift=0.6ex},->] (A) -- node[above] {$U$} (B);
\draw[transform canvas={yshift=-0.6ex},<-] (A) -- node[below] {$Q$} (B);
\end{tikzpicture}
\]
is a symmetric monoidal Quillen equivalence.
\end{theorem}

\begin{proof}
By Lemma~\ref{biequival_between_products}, the adjunction is a Quillen equivalence. Moreover, the Quillen equivalence is a symmetric monoidal Quillen equivalence since $U$ is a
strict symmetric monoidal functor.
\end{proof}

The preceding results assemble into the following commutative square of weak
symmetric monoidal Quillen equivalences, where the horizontal Quillen equivalences are strong symmetric monoidal.
\[
\begin{tikzpicture}[node distance=3cm]
\node (A) at (0,2) {$(\mathbf{Flex},\boxtimes)$};
\node (B) at (4,2) {$(2\mathbf{Cat},\boxtimes)$};
\node (C) at (0,0) {$(\mathbf{Flex},\otimes)$};
\node (D) at (4,0) {$(2\mathbf{Cat},\otimes)$};

\draw[transform canvas={yshift=0.5ex},->] (A) -- node[above] {$U$} (B);
\draw[transform canvas={yshift=-0.5ex},<-] (A) -- node[below] {$Q$} (B);

\draw[transform canvas={xshift=-0.5ex},->] (A) -- node[left] {$1$} (C);
\draw[transform canvas={xshift=0.5ex},<-] (A) -- node[right] {$1$} (C);

\draw[transform canvas={xshift=-0.5ex},->] (B) -- node[left] {$1$} (D);
\draw[transform canvas={xshift=0.5ex},<-] (B) -- node[right] {$1$} (D);

\draw[transform canvas={yshift=0.5ex},->] (C) -- node[above] {$U$} (D);
\draw[transform canvas={yshift=-0.5ex},<-] (C) -- node[below] {$Q$} (D);
\end{tikzpicture}
\]

\noindent \textbf{Acknowledgments} We would like to thank my husband and Tony Elmendorf for constant encouragement. 

\noindent \textbf{Author Contributions} JT carried out the main research and wrote the main paper. 

\noindent \textbf{Funding} The authors did not receive support from any organization for the submitted work.

\noindent \textbf{Data Availability} No datasets were generated or analyzed during the current study.
Declarations
C

\noindent \textbf{Competing interests} The authors declare no competing interests.


\begin{thebibliography}{99}
\bibitem[Bourke and Gurski, 2015]{BourkeGurski2015}
J. Bourke and N. Gurski,
\emph{A cocategorical obstruction to tensor products of Gray-categories}.
Theory and Applications of Categories \textbf{30} (2015), No.~11, 387--409.

\bibitem[Bourke and Gurski, 2017]{BourkeGurski}
J. Bourke and N. Gurski,
\emph{The Gray tensor product via factorisation}.
Applied Categorical Structures \textbf{25} (2017), no.~4, 603--624.

\bibitem[Campbell, 2026]{Campbell2026}
A. Campbell,
\emph{A Convenient Model Category for Bicategories}.
arXiv:2606.03200, 2026.

\bibitem[Foltz et al., 1980]{FoltzLairKelly1980}
F. Foltz, C. Lair, and G.~M. Kelly,
\emph{Algebraic categories with few monoidal biclosed structures or none}.
Journal of Pure and Applied Algebra \textbf{17} (1980), no.~2, 171--177.

\bibitem[Gray, 1974]{Gray1974}
J.~W. Gray,
\emph{Formal Category Theory: Adjointness for $2$-Categories}.
Lecture Notes in Mathematics, Vol.~391,
Springer-Verlag, Berlin--New York, 1974.

\bibitem[Lack, 2002]{Lack2002}
S. Lack,
\emph{A Quillen model structure for $2$-categories}.
K-Theory \textbf{26} (2002), no.~2, 171--205.

\bibitem[Lack, 2006]{Lack2006}
S. Lack,
\emph{A convenient 2-category of bicategories}.
Talk presented at CT2006, the International Category Theory Conference,
2006. Available at
\url{https://www.mscs.dal.ca/~selinger/ct2006/slides/CT06-Lack.pdf}.

\bibitem[Schwede and Shipley, 2003]{SchwedeShipley2003}
S. Schwede and B.~E. Shipley,
\emph{Equivalences of monoidal model categories}.
Algebraic \& Geometric Topology \textbf{3} (2003), 287--334.
\end{thebibliography}
\end{document}